\documentclass[12pt, reqno]{amsart}
\usepackage{amsmath, amsthm, amscd, amsfonts, amssymb, graphicx, color}
\usepackage[bookmarksnumbered, colorlinks, plainpages]{hyperref}
\hypersetup{colorlinks=true,linkcolor=red, anchorcolor=green, citecolor=cyan, urlcolor=red, filecolor=magenta, pdftoolbar=true}

\usepackage{tensor}

\newtheorem{theorem}{Theorem}[section]

\newtheorem{proposition}[theorem]{Proposition}
\newtheorem{corollary}[theorem]{Corollary}
\theoremstyle{definition}
\newtheorem{definition}[theorem]{Definition}
\newtheorem{example}[theorem]{Example}

\theoremstyle{remark}
\newtheorem{remark}[theorem]{Remark}
\numberwithin{equation}{section}

\usepackage{fullpage}

\begin{document}

\setcounter{page}{1}

\title[FPT on $\varepsilon$-chainable spaces]{Fixed point theorems for $(\varepsilon,\lambda)$-uniformly locally contractive mapping defined on
$\varepsilon$-chainable $G$-metric type spaces.}

\author[Ya\'e Olatoundji Gaba]{Ya\'e Olatoundji Gaba$^{1,*}$}

\address{$^{1}$Department of Mathematics and Applied Mathematics, University of Cape Town, South Africa.}
\email{\textcolor[rgb]{0.00,0.00,0.84}{gabayae2@gmail.com
}}



\begin{abstract}
In this article, we discuss fixed point results for $(\varepsilon,\lambda)$-uniformly locally contractive self mapping defined on $\varepsilon$-chainable $G$-metric type spaces. In particular, we show that under some more general conditions, certain fixed point results already obtained in the literature remain true. Moreover, in the last sections of this paper, we make use of the newly introduced notion of $\lambda$-sequences to derive new results.
\end{abstract} 

\maketitle

\section{Introduction and preliminaries}
Fixed-point theory is an important and flourishing
area of research of pure and applied mathematics. Its relevance is due to the fact that in many real life problems, it is a key mathematical tool used to establish the existence of solutions. Although the basic ideas
for fixed-point theory came from metric space topology, the last decades have seen a rapid growth of the theory in metric-type spaces, see \cite{Gaba3,Gaba2} where concepts like startpoint, endpoint were introduced, as ``fixed-point like" theory.  We also know from Mustafa \cite[Proposition 5]{Mustafa} that every $G$-metric space is topologically equivalent to a metric space but $G$-metric spaces and metric spaces are ``isometrically" distinct.

\vspace*{0.2cm}

Many fixed point in $G$-metric type spaces appear in the literature and the works by Jleli\cite{j}, Kadelburg\cite{r}, Mohanta\cite{s}, Mustafa et al. (\cite{mustafa1, mustafa3, mustafa4,mustafa2}), Patil\cite{p}, Tran Van An\cite{t} and many more, are very enlightening on the subject.
In \cite{Gaba1},  we began the study of fixed point for certain maps defined on $G$-metric type spaces. Our purpose in the present paper is to pursue this study by providing new fixed point results. 
We make use of the idea of orbitally complete and $\varepsilon$-chainable $G$-metric type spaces as well as the concept of $(\varepsilon,\lambda)$-uniformly locally contractive mapping that we introduce in this paper. We also show how the idea of $\lambda$-sequence can be used to prove some of these results.
The method builds on the convergence of an appropriate series of coefficients. Recent and similar work
can also be read in \cite{Gaba5,v}.

\vspace*{0.2cm}

We recall here some key results that will be useful in the rest of this manuscript.
The basic concepts and notations attached to the idea of $G$-metric type spaces are merely copies of those introduced for $G$-metric spaces and can be read extensively in \cite{Mustafa} but for the convenience of the reader, we here recall the most important ones.

\newpage

In \cite[Definition 6]{kham}, Khamsi and Hussain introduced the so-called metric-type
space $(X,m,\alpha)$, where the classical trianlge inequality condition is replaced by

$$m(x,y) \leq \alpha [ m(x,z_1)+m(z_1,z_2)+\cdots+m(z_n,y) ]$$ for any points for any points $x,y,z,z_i\in X,\ i=1,2,\ldots, n$ where $n\geq 1$ and some non-negative constant $\alpha\geq 0$.

Imitating this, we introduced in \cite{Gaba1} the definition below:

\begin{definition} (Compare \cite[Definition 3]{Mustafa})
Let $X$ be a nonempty set, and let the function $G:X\times X\times X \to [0,\infty)$ satisfy the following properties:
\begin{itemize}
\item[(G1)] $G(x,y,z)=0$ if $x=y=z$ whenever $x,y,z\in X$;
\item[(G2)] $G(x,x,y)>0$ whenever $x,y\in X$ with $x\neq y$;
\item[(G3)] $G(x,x,y)\leq G(x,y,z) $ whenever $x,y,z\in X$ with $z\neq y$;
\item[(G4)] $G(x,y,z)= G(x,z,y)=G(y,z,x)=\ldots$, (symmetry in all three variables);
\item[(G5)]
$$G(x,y,z) \leq K [G(x,z_1,z_1)+G(z_1,z_2,z_2)+\cdots+G(z_n,y,z)]$$ for any points $x,y,z,z_i\in X,\ i=1,2,\ldots, n$ where $n\geq 1$.
\end{itemize}
The triplet $(X,G,K)$ is called a \textbf{$G$-metric type space}.

\end{definition}

\begin{remark}(Compare \cite{Gaba1})
We can easily observe that $G$-metric type spaces generalize $G$-metric spaces and that for $K=1$, we recover the classical $G$-metric. Furthermore,  if $(X,G,K)$ is a $G$-metric type space, then for any $L \geq K$, $(X,G,L)$ is also a $G$-metric type space. 
\end{remark}

Straightforward computations lead to the following.

\begin{proposition} (Compare \cite[Proposition 6]{Mustafa})
Let $(X,G,K)$ be a $G$-metric type space. Define on $X$ the metric type $d_G$ by $d_G(x,y)= G(x,y,y)+G(x,x,y)$ whenever $x,y \in X$. Then for a sequence $(x_n) \subseteq X$, the following are equivalent
\begin{itemize}
\item[(i)] $(x_n)$ is $G$-convergent to $x\in X.$

\item[(ii)] $\lim_{n,m \to \infty}G(x,x_n,x_m)=0.$

\item[(iii)]  $\lim_{n \to \infty}d_G(x_n,x)=0.$

\item[(iv)]$\lim_{n \to \infty}G(x,x_n,x_n)=0.$ 

\item[(v)]$\lim_{n \to \infty}G(x_n,x,x)=0.$ 
\end{itemize}

\end{proposition}

\begin{proposition}(Compare \cite[Proposition 9]{Mustafa})

In a $G$-metric type space $(X,G,K)$, the following are equivalent
\begin{itemize}
\item[(i)] The sequence $(x_n) \subseteq X$ is $G$-Cauchy.

\item[(ii)] For each $\varepsilon >0$ there exists $N \in \mathbb{N}$ such that $G(x_n,x_m,x_m)< \varepsilon$ for all $m,n\geq N$.

\item[(iii)] $(x_n)$ is a Cauchy sequence in the metric type space $(X, d_G,K)$.
\end{itemize}

\end{proposition}

\begin{definition} (Compare \cite[Definition 9]{Mustafa})
 A $G$-metric type space $(X,G,K)$ is said to be $G$-complete if every $G$-Cauchy sequence in $(X,G,K)$ is $G$-convergent in $(X,G,K)$.

\end{definition}

\newpage

\section{First Results}
We begin with the following property.

\begin{definition}(\cite{mustafa1})
Let $(X,G,K)$ be a $G$-metric type space. A mapping $T:X\to X$ is called Lipschitzian if there exists $k \in \mathbb{R}$ such that
\begin{equation}\label{lipschitz}
G(Tx,Ty,Tz) \leq k G(x,y,z)
\end{equation}
for all $x,y,z \in X.$ The smallest constant $k$ which  satisfies the above inequality is called the
\textbf{Lipschitz constant of $T$}, and is denoted $Lip(T)$. In particular $T$ is a \textbf{contaction} if $Lip(T) \in [0,1)$.
\end{definition}

\begin{theorem}\label{main1}
Let $(X,G,K)$ be a $G$-complete $G$-metric type space and $T:X\to X$ be a mapping such that $T^n$ is Lipschitzian for all $n\geq 1$ and that $\sum_{n=0}^\infty Lip(T^n) < \infty.$  Then $T$ has a unique fixed point $x^* \in X$. In fact, $T$ is a Picard operator.

\end{theorem}


\begin{proof}

Let $x\in X$. For any $n,h \geq 0, $ we have

\begin{equation}
G(T^{n+h}x,T^nx,T^nx) \leq Lip(T^n)G(T^{h}x,x,x)\leq K Lip(T^n)\sum_{i=0}^{h-1}G(T^{i+1}x, T^ix,T^ix).
\end{equation}

Hence

\begin{equation}
G(T^{n+h}x,T^nx,T^nx) \leq K Lip(T^n)\left( \sum_{i=0}^{h-1}Lip(T^i) \right)G(Tx, x,x).
\end{equation}

Since  $\sum_{n=0}^\infty Lip(T^n)< \infty,$ then $\lim_{n \to \infty}Lip(T^n)=0.$ This forces $(T^nx)$ to be a $G$-Cauchy sequence. Since $X$ is $G$-complete, then $(T^nx)$ converges to some point $x^*\in X.$

\vspace*{0.3cm}

\underline{Claim 1}: \textbf{$x^*$ is a fixed point of $T.$}
On the one hand we have

\begin{align}
G(T^{n-1}x,x^*,x^*) & \leq  K \left(  G(T^{n-1}x,T^{n}x,T^{n}x)  +  G(T^{n}x,x^*,x^*)            \right) \nonumber \\
                  & \leq  K \left[ Lip(T^{n-1})\ G(x,Tx,Tx)+ G(T^{n}x,x^*,x^*)           \right],
\end{align}

hence we get

\begin{align}
G(x^*,Tx^*,Tx^*) & \leq  K  \left[ G(x^*,T^{n}x,T^{n}x) + G(T^{n}x,Tx^*,Tx^*) \right] \nonumber \\
                & \leq  K \left[ G(x^*,T^{n}x,T^{n}x)   +K Lip(T)G(T^nx,x^*,x^*) \right.  \nonumber \\
               & +  \left.  K Lip(T) Lip(T^{n-1})G(x,Tx,Tx) \right].
\end{align}

Letting $n$ tends to $\infty$, we get $G(x^*,Tx^*,Tx^*) =0$\footnote{See \cite[Proposition 1]{Mustafa}, which allows us to have $G(x,y,z)=0 \Longleftrightarrow x=y=z$}, i.e $Tx^* =x^* $.

\vspace*{0.3cm}

\underline{Claim 2}: \textbf{$x^*$ is the only fixed point of $T.$}
If $a^*$ is a fixed point of $T$, then

\begin{equation}
G(a^*,x^*,x^*) \leq G(T^{n}a^*,T^{n}x^*,T^{n}x^*) \leq Lip(T^{n})G(a^*,x^*,x^*)
\end{equation}
for any $n\geq 1$. Since $\lim_{n \to \infty}Lip(T^n)=0$, we obtain that $a^*=x^*$.

\end{proof}


Instead of the property (G5), a more natural condition is what appears in \cite[Definition 3]{Mustafa}

$$(G5') \quad D(x,y,z) \leq K [D(x,z_1,z_1)+D(z_1,y,z)]$$ for any points $x,y,z,z_1\in X$ for some constant $K>0$.


\begin{theorem}\label{main2}
Let $(X,G,K)$ be a $G$-complete $G$-metric type space where $G$ satisfies (G5') instead of (G5). Let $T:X\to X$ be a mapping such that $T^n$ is Lipschitzian for all $n\geq 1$ and that $\lim_{n\to \infty} Lip(T^n) =0.$  Then $T$ has a unique fixed point $x^* \in X$ if and only if the orbit $\{T^nx, n\geq 1\}$ is bounded \footnote{Recall that a subset $Y$ of $X$ is said to be bounded whenever $\sup\{G(x,y,z), x,y,z \in X\}<\infty .$} for some $x\in X$. In fact, if there exists  $x^*$ such that $Tx^*= x^*$, then $T$ is a Picard operator.

\end{theorem}

\begin{proof}
It is clear that when $T$ has a fixed point, say $u\in X$, then its orbit $\{T^nu, n\geq 1\}=\{u\}$ is bounded.

Now let $x\in X$ and assume that the orbit $\{T^nx, n\geq 1\}$ is bounded, i.e. there exists $\alpha\geq 0$ such that $G(T^{n+h}x, T^nx,T^nx) \leq \alpha$ for any $n,h \geq 0$. Hence, we have

$$ G(T^{n+h}x, T^nx,T^nx ) \leq Lip(T^n) G(T^hx,x,x) \leq Lip(T^n) \alpha.$$

Since $\lim_{n\to \infty} Lip(T^n) =0,$ then $(T^nx)$ is a $G$-Cauchy sequence, hence $(T^nx)$ converges to some $x^*$ as $X$ is $G$-complete. The remaining part of the proof follows the same idea as in Theorem \ref{main1}.

\end{proof}

The two above results generalise the ones appearing in \cite{Mustafa}, in the sense that the mapping $T$ involved does not have to be a contraction, hence the condition on $\{Lip(T^n)\}$.

\section{Main results}

Let $(X,G,K)$ be a $G$-metric type space. 
We present here a few fixed point results for $(\varepsilon,\lambda)$-uniformly locally contractive mapping defined on $X$. We begin with the following definitions.

\begin{definition} 
Let $(X,d_1,k_1)$ and $(Y,d_2,k_2)$ be two $G$-metric type spaces. A mapping $T : X\to Y$ is said to be \textbf{sequentially continuous} if the sequence $\{Tx_n\}$ $d_2$-converges to $Tx^*$  whenever the sequence $\{x_n\}$ $d_1$-converges to $x^*$.
\end{definition}

\begin{definition} 
A self mapping $T$ defined on a $G$-metric type space $(X,G,K)$ is said to be \textbf{orbitally continuous} if and only if $\lim_{i\to \infty}T^{n_i}x=x^* \in X$ implies $Tx^*=\lim_{i\to \infty}TT^{n_i}x$.
\end{definition}

\begin{definition}
Let $T$ be a self mapping defined on a $G$-metric type space $(X,G,K)$. The space $(X,G,K)$ is said to be $T$-\textbf{orbitally complete} if and only if for any $a\in X$ every $G$-Cauchy sequence which is contained in $\{a,Ta,T^2a,T^3a, \cdots\}$ $G$-converges in $X$.
\end{definition}

\begin{remark}
We know that the Banach contraction requires the original space to be complete but this assumption can be difficult to realise and often we just need the convergence of a specific type of sequences, namely the ones generated by the orbits and for this, the idea of orbitally completeness comes as a substitute for the metric completeness. For instance if we let $X=[0,\infty)$ with $G(x,y,z)=\max\{|x-y|,|y-z|,|z-x|\}$ and define on $X$ the map $T:X\to X$ by $Tx=x(x+1)^{-1}$. Then one can convince oneself that even though $(X,G)$ is not complete, $T$ is orbitally continuous and $X$ is $T$-orbitally complete.
\end{remark}

\begin{remark}
Theorems \ref{main1} and \ref{main2} remain true if instead of requiring the space $(X,G,K)$ to be $G$-complete, we just assume that $(X,G,K)$ is $T$-orbitally complete and $T$ orbitally continuous.
\end{remark}

\begin{definition} 
Let $(X,G,K)$ be a $G$-metric type space. For $x,y\in X, x\neq y$, a \textbf{path} from $x\in X$ to $y\in X$ is a finite sequence $\{x_0,x_1,\cdots,x_n\}, n\geq 1$ of distinct points of $X$ such that $x= x_0$ and $y=x_n$. In this case, $n$ will be called the \textbf{degree}\footnote{Of course, every point is a path of degree $0$.} of the path.
\end{definition}

\begin{definition} 
The $G$-metric type space $(X,G,K)$ will be called $\varepsilon$\textbf{-chainable} for some $\varepsilon>0$ if for any two points $x,y\in X, x\neq y$, there exists a path $\{x=x_0,x_1,\cdots,y=x_n\}$ from $x$ to $y$ such that $G(x_i, x_{i+1},x_{i+1})\leq \varepsilon$ for $i=0,1,\cdots,n-1$.
\end{definition}

\begin{example}
Let $X$ be a non-empty set. We endow $X$ with the \textit{discrete} $G$-metric:

$$
G(x,y,z)=
\begin{cases}
0, \text{ if } x=y=z,\\
1, \text{ otherwise}.
\end{cases}
$$

$(X,G,1)$ is $1$-chainable but not $\frac{1}{2}$-chainable.
\end{example}

\begin{definition} 
Let $(X,G,K)$ be a $G$-metric type space. A self mapping $T$ defined on $X$ is called \textbf{locally contractive} if for every $x\in X$, there exist $\varepsilon_x\geq 0$ and $\lambda_x \in [0,1)$ such that 

\begin{equation}
G(Tu, Tv,Tp) \leq \lambda_x G(u,v,p) 
\end{equation}
whenever $u,v,p \in C_G(x,\varepsilon_x):= \{ y: G(x,y,y) \leq \varepsilon_x\}$.
\end{definition}

\begin{definition} 
Let $(X,G,K)$ be a $G$-metric type space. A self mapping $T$ on $X$ is called \textbf{uniformly locally contractive} if it is locally contractive and for every $x,y\in X, x\neq y$, $\varepsilon:=\varepsilon_x=\varepsilon_y\geq 0$ and $\lambda:=\lambda_x=\lambda_y \in [0,1)$, i.e. the constants $\varepsilon_x, \lambda_x$ do not depend on the choice of $x\in X$.
\end{definition}

\subsection{The sequential condition}

\begin{theorem}\label{seq1}
Let $(X,G,K)$ be a $G$-complete $G$-metric type space and let $T$ be a sequentially continuous\footnote{Or just orbitally continuous.} self mapping on $X$ such that

\begin{equation}\label{sequel1}
G(T^nx,T^ny,T^nz) \leq a_n [ G(x,Tx,Tx)+G(y,Ty,Ty)+G(z,Tz,Tz) ]
\end{equation}
for all $x,y,z \in X$ where $a_n(>0)$ for all $n\geq 1$ are independent from $x,y,z$ and $0\leq a_1 <\frac{1}{2}$. If the series $\sum a_n$\footnote{It is enough that the sequence $(a_n)$ converges to $0$} is convergent, then $T$ has a unique fixed point in $X$.

\end{theorem}

\begin{proof}

Let $x_0\in X$. We consider the sequence of iterates $x_n = T^nx_0,n = 1,2,3,\cdots.$ Then for $n\geq 1$

\begin{align*}
G(T^nx_0,T^{n+1}x_0,T^{n+1}x_0 )  \leq  & a_n [G(x_0, Tx_0,Tx_0)+ G(Tx_0, T^2x_0,T^2x_0) \\ & + G(Tx_0, T^2x_0,T^2x_0)] \\
  &=    a_n [G(x_0, Tx_0,Tx_0)+ 2 G(Tx_0, T^2x_0,T^2x_0)].
\end{align*}


 Again 
 
 \begin{align*}
 G(Tx_0, T^2x_0,T^2x_0) \leq  & a_1 [G(x_0, Tx_0,Tx_0)+ G(Tx_0, T^2x_0,T^2x_0) \\ & + G(Tx_0, T^2x_0,T^2x_0)].    
\end{align*}
 
Therefore 

\begin{equation}\label{cond1}
G(T^nx_0,T^{n+1}x_0,T^{n+1}x_0 ) \leq a_n \left[1+ \frac{2 a_1}{1-2a_1} \right] G(x_0, Tx_0,Tx_0).
\end{equation}

Using property (G5), we can write:

\begin{align*}
G(x_n,x_{n+m},x_{n+m} ) & = G(T^nx_0,T^{n+m}x_0,T^{n+m}x_0 ) \\
                    \leq & K [G(T^nx_0,T^{n+1}x_0,T^{n+1}x_0 ) + G(T^{n+1}x_0, T^{n+2}x_0,T^{n+2}x_0) + \\
                   & + \cdots +G( T^{n+m-1}x_0,T^{n+m}x_0,T^{n+m}x_0)].
\end{align*}

So using \eqref{cond1}, we get

\begin{align*}
G(x_n,x_{n+m},x_{n+m} ) & \leq [a_n+a_{n+1}+\cdots+a_{n+m-1}]\left[1+ \frac{2 a_1}{1-2a_1} \right]G(x_0, Tx_0,Tx_0).
\end{align*}

Now since $\sum a_n$ is convergent, we get that $G(x_n,x_{n+m},x_{n+m} )\to 0$ as $n\to \infty$ and the sequence $(x_n)$ is $G$-Cauchy. Moreover, since $X$ is $G$-complete and $T$ sequentially continuous, there exists $x^* \in X$ such that $(x_n)$ $G$-converges to $x^*$ and $(x_{n+1})$ $G$-converges to $Tx^*=x^*$ because $(X,G,K)$ is Hausdorff. If $x^*, z^*$ are fixed points for $T$, then from \eqref{sequel1}, we have $x^*=T^nx^*=T^nz^*= z^*, \forall n\geq 1,$ i.e. $T$ has a unique fixed point.

\end{proof}

In a similar way, one can establish that: 

\begin{theorem}\label{seq2}
Let $(X,G,K)$ be a $G$-complete $G$-metric type space and let $T$ be a sequentially continuous self mapping such that

\begin{align*}
G(T^nx,T^ny,T^nz) \leq a_n [ G(x,Tx,Tx)+G(y,Ty,Ty)+G(z,Tz,Tz) ]
\end{align*}
for all $x,y,z \in X$ where $a_n(>0)$ for all $n\geq 1$ are independent from $x,y,z$ and $0\leq a_1 <\frac{1}{2}$. If $X$ is $T$-orbitally complete and that the series $\sum a_n$ is convergent, then $T$ has a unique fixed point in $X$.

\end{theorem}

The next two results are inspired by Theorem \ref{main2}.

\begin{theorem}\label{khma1}
Let $(X,G,K)$ be a $G$-complete $G$-metric type space and let $T$ be an orbitally continuous self mapping on $X$ such that

\begin{equation}\label{khamsi1}
G(T^nx,T^ny,T^nz) \leq a_n [ G(x,Tx,Tx)+G(y,Ty,Ty)+G(z,Tz,Tz) ]
\end{equation}
for all $x,y,z \in X$ where $a_n(>0)$ for all $n\geq 1$ are independent from $x,y,z$ and $0\leq a_1 <\frac{1}{2}$. We assume that $\lim_{n\to \infty} a_n=0$. Then $T$ has a unique fixed point $x^* \in X$ if and only if the orbit $\{T^nx, n\geq 1\}$ is bounded for some $x\in X$. In fact, if there exists  $x^*$ such that $Tx^*= x^*$, then $T$ is a Picard operator.

\end{theorem}

\begin{theorem}\label{khma2}
Let $T$ be an orbitally continuous self mapping on a $T$-orbitally complete $G$-metric type space  $(X,G,K)$ such that

\begin{equation}\label{khamsi2}
G(T^nx,T^ny,T^nz) \leq a_n [ G(x,Tx,Tx)+G(y,Ty,Ty)+G(z,Tz,Tz) ]
\end{equation}
for all $x,y,z \in X$ where $a_n(>0)$ for all $n\geq 1$ are independent from $x,y,z$ and $0\leq a_1 <\frac{1}{2}$. We assume that $\lim_{n\to \infty} a_n=0$. Then $T$ has a unique fixed point $x^* \in X$ if and only if the orbit $\{T^nx, n\geq 1\}$ is bounded for some $x\in X$. In fact, if there exists  $x^*$ such that $Tx^*= x^*$, then $T$ is a Picard operator.

\end{theorem}

\subsection{The $\Phi$-class extension}\hspace*{0.2cm}

Let $\Phi$ be the class of continuous, non-decreasing, sub-additive and homogeneous functions $F:[0,\infty) \to [0,\infty)$ such that $F^{-1}(0)=\{0\}$. We have the following interesting result which generalises Theorem \ref{seq1}.

\begin{theorem}\label{seq11}
Let $(X,G,K)$ be a $G$-complete $G$-metric type space and let $T$ be a sequentially continuous self mapping such that

\begin{equation}\label{free1}
F(G(T^nx,T^ny,T^nz)) \leq F(a_n [ G(x,Tx,Tx)+G(y,Ty,Ty)+G(z,Tz,Tz) ])
\end{equation}
for all $x,y,z \in X$ where $a_n(>0)$ is independent from $x,y,z$ and $0\leq a_1<\frac{1}{2}$ for some $F\in \Phi$ homogeneous with degree $s$. If the series $\sum a_n$ is convergent, then $T$ has a unique fixed point in $X$. Moreover $T$ is a Picard operator.

\end{theorem}

\begin{proof}
Let $x_0\in X$. We consider the sequence of iterates $x_n = T^nx_0,n = 1,2,3,\cdots.$ Then for $n\geq 1$

\begin{align*}
F(G(T^nx_0,T^{n+1}x_0,T^{n+1}x_0 )) \leq & F(a_n [G(x_0, Tx_0,Tx_0)+ 2 G(Tx_0, T^2x_0,T^2x_0) ])\\
  \leq & a_n^s F(G(x_0, Tx_0,Tx_0)) + (2a_n)^s F(G(Tx_0, T^2x_0,T^2x_0)) ]
\end{align*}

Again

\begin{align*}
 F(G(Tx_0, T^2x_0,T^2x_0)) \leq  & F(a_1 [G(x_0, Tx_0,Tx_0)+ 2G(Tx_0, T^2x_0,T^2x_0)])\\
     \leq  & a_1^s F(G(x_0, Tx_0,Tx_0)) + (2a_1)^s F( G(Tx_0, T^2x_0,T^2x_0)).
\end{align*}
 which gives
 
\[
F(G(Tx_0, T^2x_0,T^2x_0)) \leq \frac{a_1^s}{1-(2a_1)^s} F(G(x_0, Tx_0,Tx_0))
\]

Therefore 

\begin{equation}\label{cond2}
FG(T^nx_0,T^{n+1}x_0,T^{n+1}x_0 )) \leq a_n^s \left[ 1 + \frac{(2a_1)^s}{1-(2a_1)^s } \right]F(G(x_0, Tx_0,Tx_0)).
\end{equation}

Using property (G5) and \ref{cond2}, we can write:

\begin{align*}
F(G(x_n,x_{n+m},x_{n+m} )) & \leq [a_n^s+a_{n+1}^s+\cdots+a_{n+m-1}^s]\left[1+\frac{(2a_1)^s}{1-(2a_1)^s}\right] F(G(x_0, Tx_0,Tx_0)).
\end{align*}
As $n \to \infty $, since $F^{-1}(0) = {0}$ and $F$ is continuous, we deduce that $G(x_n,x_{n+m},x_{n+m} )\to 0$ and the sequence $(x_n)$ is $G$-Cauchy. Moreover, since $X$ is $G$-complete and $T$ sequentially continuous, there exists $x^* \in X$ such that $(x_n)$ $G$-converges to $x^*$ and $x_{n+1}$ $G$-converges to $Tx^*=x^*$ because $(X,G,K)$ is Hausdorff.
The uniqueness of $x^*$ is given for free by the condition \eqref{free1}.

\end{proof}

In a similar way, one can establish that 

\begin{theorem}\label{seq12}
Let $(X,G,K)$ be a $G$-complete $G$-metric type space and let $T$ be a sequentially continuous self mapping such that

\begin{equation}
F(G(T^nx,T^ny,T^nz)) \leq F(a_n [ G(x,Tx,Tx)+G(y,Ty,Ty)+G(z,Tz,Tz) ])
\end{equation}
for all $x,y,z \in X$ where $a_n(>0)$ is independent from $x,y,z$ and $0\leq a_1 <\frac{1}{2}$ and for some $F\in \Phi$ homogeneous with degree $s$. If $X$ is $T$-orbitally complete and that the series $\sum a_n$ is convergent, then $T$ has a unique fixed point in $X$. 

\end{theorem}

\begin{remark}
If we set $F=Id_{[0,\infty)}$ in the Theorem \ref{seq1}, we obtain the result of Theorem \ref{seq1}; the same applies to Theorem \ref{seq2} with regard to Theorem \ref{seq12}.
\end{remark}


\begin{example}
Let $X=[0,1]$ and $G(x,y,z)= \max\{x,y,z \}$ whenever $x,y,z \in [0,1]$. Clearly, $(X,G,1)$ is a $G$-complete $G$-metric space. 

 Following the notation in Theorem \ref{seq11}, we set $a_n = \left(\frac{1}{1+2^n}\right)^2.$ 
 
 We also define $T(x)= \frac{x}{16}$ for all $x\in [0,1]$ and let $F$ be defined as $F:[0,\infty) \to [0,\infty), \ x\mapsto \sqrt{x}$. Then $F$ is continuous, non-decreasing, sub-additive and homogeneous of degree $s=\frac{1}{2}$ and $F^{-1}(0)=\{0\}$. Assume $x>y\geq z$. Hence we have

\[ F(G(T^nx,T^ny,T^nz)) = \sqrt{\frac{x^n}{16^n}}\leq \sqrt{\frac{x}{16^n}}, \]
and
 
\[ F( a_n [G(x,Tx,Tx)+ G(y,Ty,Ty) + G(z,Tz,Tz)]  ) =  \sqrt { \left(\frac{1}{1+2^n}\right)^2 (x+y+z)} .\]

Observe that $\sum a_n \leq \sum \frac{1}{n^2} < \infty$ and $a_1= \frac{1}{9}< \frac{1}{2}$. The conditions of Theorem \ref{seq11} are satisfied, so $T$ has a unique fixed point, which in this case is $x^*=0.$

\end{example}

A more general result can be written as:

\begin{theorem}\label{mainmore}
Let $T$ be an orbitally continuous self mapping on a $T$-orbitally complete $G$-metric type space  $(X,G,K)$ such that

\begin{equation}\label{mainmore1}
F(G(T^nx,T^ny,T^nz)) \leq F(a_n [ G(x,Tx,Tx)+G(y,Ty,Ty)+G(z,Tz,Tz) ])
\end{equation}
for all $x,y,z \in X$ where $a_n(>0)$ for all $n\geq 1$ are independent from $x,y,z$ and $0\leq a_1 <\frac{1}{2}$ and for some $F\in \Phi$ homogeneous with degree $s$. We assume that $\lim_{n\to \infty} a_n=0$. Then $T$ has a unique fixed point $x^* \in X$ if and only if the orbit $\{T^nx, n\geq 1\}$ is bounded for some $x\in X$. In fact, if there exists  $x^*$ such that $Tx^*= x^*$, then $T$ is a Picard operator.

\end{theorem}

\subsection{The $\varepsilon$-chainable setting} \hspace*{0.2cm}

The next result illustrates the use of the $\varepsilon$-chainability in fixed point theory. 

The formulation is given as follows:

\begin{theorem}\label{seqL1}

If $T$ is a $(\varepsilon,\lambda)$-uniformly locally contractive and orbitally continuous mapping defined on a $T$-orbitally complete and 
$\frac{\varepsilon}{2}$-chainable $G$-metric type space $(X,G,K)$, then $T$ has a unique fixed point.

\end{theorem}

\begin{proof}
Let $x\in X$. If $Tx=x$, then we are done. Else, since $X$ is $\frac{\varepsilon}{2}$-chainable, there exists a path $\{x=x_0,x_1,\cdots,x_n=Tx_0\}$ from $x$ to $Tx$ such that 

$$G(x_i, x_{i+1},x_{i+1})\leq \frac{\varepsilon}{2}$$ for $i=0,1,\cdots,n-1$. It is very clear that 

\[ G(x,Tx,Tx) \leq \frac{K \varepsilon}{2}.\]

Since $T$ is $(\varepsilon,\lambda)$-uniformly locally contractive, 
\[
G(Tx_i, Tx_{i+1},Tx_{i+1}) \leq \lambda G(x_i, x_{i+1},x_{i+1}) <  \frac{\lambda \varepsilon}{2} \quad \forall i=0,1,\cdots,n-1.
\]

Hence, by induction 

\[
G(T^mx_i, T^mx_{i+1},T^mx_{i+1}) \leq \lambda^m G(x_i, x_{i+1},x_{i+1}) <  \frac{\lambda^m \varepsilon}{2} \quad \forall m\geq 1.
\]

Moreover, by property (G5)

\begin{align*}
G(T^mx,T^{m+1}x,T^{m+1}x ) \leq & K[G(T^mx,T^{m}x_1,T^{m}x_1 )+ G(T^mx_1,T^{m}x_2,T^{m}x_2 ) + \cdots \\ 
& + G(T^mx_{n-1},T^{m}Tx,T^{m}Tx ) ], 
\end{align*}
and the above induction, we conclude that 

\[
G(T^mx,T^{m+1}x,T^{m+1}x ) \leq \frac{ \lambda^m K n \varepsilon}{2} \quad \forall m\geq 1.
\]

For $l,m\geq 1 $, and again from property (G5)

\[
G(T^mx,T^{m+l}x,T^{m+l}x ) <  \frac{\lambda^m}{1-\lambda}  \frac{  K^2 n \varepsilon}{2},
\]

which establishes that $\{T^nx\}\subseteq \{x,Tx,T^2x,\cdots\}$ is a $G$-Cauchy sequence and $G$-converges to some $x^* \in X$ since $X$ is $T$-orbitally complete. Obviously $x^*$ is the desired fixed point by orbitally continuity of $T$.

For uniqueness, if $z^*$ is a fixed point such that $x^* \neq z^*$, we can find a path or an $\frac{\varepsilon}{2}$-chain, from $x^*$ to $z^*$ 
with
\[
x^*=x_0,x_1, \cdots,x_n=z^*,
\]
We know that

\[
G(T^mx^*,T^mz^*,T^mz^*) < \frac{\lambda^m K n  \varepsilon}{2} 
\quad \forall m\geq 1.
\]

\newpage

Hence

\[
G(x^*,z^*,z^*)=G(T^mx^*,T^mz^*,T^mz^*) < \frac{\lambda^m K n  \varepsilon}{2} 
\quad \forall m\geq 1.
\]
As $m\to \infty$ $G(x^*,z^*,z^*)=0$ and $x^*=z^*$.
\end{proof}

We conclude this subsection with these examples.

\begin{example}
Let $X = \{0,1\}$ be endowed with the $G$-metric:

\[
G(0, 0, 0)=G(1, 1, 1)=0; \ G(0, 0, 1)=1; \ G(0, 1, 1)=2.
\]
Let $T:X\to X$ be the mapping $T0=T1=0$.

It is easy to see that $(X,G,1)$ is $\frac{ \varepsilon}{2}$-chainable with $\varepsilon=4$ and $T$-orbitally complete\footnote{The $G$-Cauchy sequences in the orbits are actually stationary}. It can also be noticed that $T$ is $(\varepsilon,\lambda)$-uniformly locally contractive with $\lambda = \frac{1}{2}$ and $T$ has a unique fixed point $x^*=0$.
\end{example}

\begin{example}
Let $X = \{0,1,2\}$ be endowed with the $G$-metric:

\begin{align*}
& G(0, 0, 0)=G(1, 1, 1)=G(2,2,2)=0; \\
& G(0, 0, 1)=G(1, 1, 0)=1;\\
& G(0, 0, 2)=G(1, 1, 2)=G(0, 2, 2)=G(1, 2, 2)=G(0, 1, 2)=2.
\end{align*}

Let $T:X\to X$ be the mapping 

\[
T(x)=\begin{cases}
0, \text{ if } x=0,1,\\
1, \text{ if } x=2.
\end{cases}
\]

$(X,G,1)$ is $\frac{ \varepsilon}{2}$-chainable with $\varepsilon=4$ and $T$-orbitally complete. Also notice that $T$ is $(\varepsilon,\lambda)$-uniformly locally contractive with $\lambda = \frac{1}{2}$ and $T$ has a unique fixed point $x^*=0$.
\end{example}

\begin{remark}
In general, for a bounded $G$-metric type space $(X,G,K)$, if we set $$\delta:=\sup\{G(x,y,z), x,y,z \in X\},$$ then $X$ is $\delta$-chainable.
\end{remark}

\subsection{The $\lambda$-sequences setting} \hspace{1.2cm}

\vspace*{0.3cm}
This section aims at making use of the technique behind the $\lambda$-sequences in fixed point theory.
 In \cite{Gaba1, Gaba4}, we introduced the idea of $\lambda$-sequences in $G$-metric type spaces as follow:

\begin{definition}\label{def1} (Compare \cite{Gaba4})
A sequence $(x_n)_{n\geq 1}$ in a $G$-metric type space $(X,G,K)$ is a $\lambda$-sequence if there exist $0<\lambda<1$ and $n(\lambda) \in \mathbb{N}$ such that $$\sum_{i=1}^{L-1} G(x_i,x_{i+1},x_{i+1}) \leq \lambda L \text{ for each } L\geq n(\lambda)+1.$$

\end{definition}
\begin{theorem}\label{lambda1}
Let $(X,G,K)$ be a $G$-complete $G$-metric type space, and let $T$ be a sequentially continuous self mapping on $X$. Assume that there exist two sequences $(a_n)$ and $(b_n)$ of elements of $X$ such 

\begin{align}\label{condition1}
F(G(T^i(x),T^j(y),T^kj(z))) \leq \  
& F(\tensor*[_{k}]{\Delta}{_{i,j}}[G(x,T^ix,T^ix)+ G(y,T^jy,T^jy) + G(z,T^kz,T^kz)])\\
                    &+ F(\tensor*[_{k}]{\Gamma}{_{i,j}}G(x,y,z)) \nonumber
\end{align}

for $x,y,z\in X$ with $x\neq y,$ $0\leq \tensor*[_{k}]{\Delta}{_{i,j}},\tensor*[_{k}]{\Gamma}{_{i,j}}<1/2 , \ i,j = 1,2,\cdots ,$ and some $F \in \Phi,$ homogeneous with degree $s$, where $\tensor*[_{k}]{\Delta}{_{i,j}}=G(a_i,a_j,a_k)$ and $\tensor*[_{k}]{\Gamma}{_{i,j}}=G(b_i,b_j,b_k)$.
If the sequence $(r_i)$ where  $$r_i=\frac{(\tensor*[_{i+1}]{\Delta}{_{i,i+1}})^s+(\tensor*[_{i+1}]{\Gamma}{_{i,i+1}})^s}{1-(2\tensor*[_{i+1}]{\Delta}{_{i,i+1}})^s}$$ is a non-increasing $\lambda$-sequence of $\mathbb{R}^+$ endowed with the $\max\footnote{The max metric $m$ refers to $m(x,y)=\max\{x,y\}$ }$ metric, then $T$ has a unique fixed point in $X$.

\end{theorem}

\begin{proof}
Let $x_0\in X$, we construct the sequence $(x_n)$ by setting $x_n= T(x_{n-1}),\ n=1,2,\cdots .$ Notice that if there is $n_0 \in \mathbb{N}$ such that $x_{n_0}=x_{n_0+1}$, then obviously $T$ has a fixed
point. Thus, suppose that 
\[
x_{n}\neq x_{n+1} \ \text{for any }  n\geq 0.
\]

Using \eqref{condition1} and the homogeneity of $F$, we obtain
\begin{align*}
F(G(x_1,x_2,x_2))  & = \  F(G(Tx_0,TTx_0,TTx_0)) \\
              &  \leq \     (\tensor*[_{2}]{\Delta}{_{1,2}})^s 
               F([G(x_0,x_1,x_1))+ 2G(x_1,x_2,x_2)]) \\
               & + (\tensor*[_{2}]{\Gamma}{_{1,2}})^s   F(G(x_0,x_1,x_1)).
 \end{align*}

Therefore

\[
F(G(x_1,x_2,x_2))  \leq \left( \frac{(\tensor*[_{2}]{\Delta}{_{1,2}})^s+(\tensor*[_{2}]{\Gamma}{_{1,2}})^s   }{1-(2\tensor*[_{2}]{\Delta}{_{1,2}})^s} \right)F(G(x_0,x_1,x_1)).
\]

Also, we get

\begin{align*}
F(G(x_2,x_3,x_3))  & \leq \left[\frac{(\tensor*[_{3}]{\Delta}{_{2,3}})^s+(\tensor*[_{3}]{\Gamma}{_{2,3}})^s }{1-(2\tensor*[_{3}]{\Delta}{_{2,3}})^s}\right] F([G(x_1,x_2,x_2)\\
& =   \left[\frac{(\tensor*[_{3}]{\Delta}{_{2,3}})^s+(\tensor*[_{3}]{\Gamma}{_{2,3}})^s }{1-(2\tensor*[_{3}]{\Delta}{_{2,3}})^s}\right]\left[ \frac{(\tensor*[_{2}]{\Delta}{_{1,2}})^s+(\tensor*[_{2}]{\Gamma}{_{1,2}})^s   }{1-(2\tensor*[_{2}]{\Delta}{_{1,2}})^s} \right]   F([G(x_0,x_1,x_1).
\end{align*}

By repeating the above process, we have

\begin{equation}
F(G(x_n,x_{n+1},x_{n+1})) \leq \prod\limits_{i=1}^n \left(   \frac{(\tensor*[_{i+1}]{\Delta}{_{i,i+1}})^s+(\tensor*[_{i+1}]{\Gamma}{_{i,i+1}})^s}{1-(2\tensor*[_{i+1}]{\Delta}{_{i,i+1}})^s} \right)  F(G(x_0,x_1,x_1)).
\end{equation}

\newpage

Hence we derive, by making use of the the property (G5) and the properties of $F$, that for $p>0$

\begin{align*}
 F(G(x_n,x_{n+p},x_{n+p})) \leq \ & K^{s}[F(G(x_n,x_{n+1},,x_{n+1})) +F(G(x_{n+1},x_{n+2},x_{n+2})) \\
                         & + \ldots + F(G(x_{n+p-1},x_{n+p},x_{n+p}))]  \\
                  \leq \ & K^{s}\left[ \prod\limits_{i=1}^n \left(   \frac{\tensor*[_{i+1}]{(\Delta}{_{i,i+1}})^s+(\tensor*[_{i+1}]{\Gamma}{_{i,i+1}})^s}{1-(2\tensor*[_{i+1}]{\Delta}{_{i,i+1}})^s} \right)  F(G(x_0,x_1,x_1))\right. \\
                  & + \prod\limits_{i=1}^{n+1} \left(   \frac{\tensor*[_{i+1}]{(\Delta}{_{i,i+1}})^s+(\tensor*[_{i+1}]{\Gamma}{_{i,i+1}})^s}{1-(2\tensor*[_{i+1}]{\Delta}{_{i,i+1}})^s} \right) F(G(x_0,x_1,x_1))      \\
                  & + \ldots + \\
                  & + \left.\prod\limits_{i=1}^{n+p-1} \left(  \frac{(\tensor*[_{i+1}]{\Delta}{_{i,i+1}})^s+(\tensor*[_{i+1}]{\Gamma}{_{i,i+1}})^s}{1-(2\tensor*[_{i+1}]{\Delta}{_{i,i+1}})^s} \right)  F(G(x_0,x_1,x_1))\right] \\
                  = \ & K^{s}\left[ \sum\limits_{k=0}^{p-1} \prod\limits_{i=1}^{n+k} \left( \frac{(\tensor*[_{i+1}]{\Delta}{_{i,i+1}})^s+(\tensor*[_{i+1}]{\Gamma}{_{i,i+1}})^s}{1-(2\tensor*[_{i+1}]{\Delta}{_{i,i+1}})^s}  \right) F(G(x_0,x_1,x_1))\right] \\
                  = \ & K^{s}\left[\sum\limits_{k=n}^{n+p-1} \prod\limits_{i=1}^{k} \left( \frac{(\tensor*[_{i+1}]{\Delta}{_{i,i+1}})^s+(\tensor*[_{i+1}]{\Gamma}{_{i,i+1}})^s}{1-(2\tensor*[_{i+1}]{\Delta}{_{i,i+1}})^s}  \right) F(G(x_0,x_1,x_1))\right].
\end{align*}

Now, let $\lambda$ and $n(\lambda)$ as in Definition \ref{def1}, then for $n\geq n(\lambda)$ and using the fact the geometric mean of non-negative real numbers is at most their arithmetic mean, it follows that

\begin{align}
F(D(x_n,x_{n+p},x_{n+p})) \leq \ & K^{s}\left[\sum\limits_{k=n}^{n+p-1} \left[ \frac{1}{k}\sum\limits_{i=1}^{k} \left( \frac{(\tensor*[_{i+1}]{\Delta}{_{i,i+1}})^s+(\tensor*[_{i+1}]{\Gamma}{_{i,i+1}})^s}{1-(2\tensor*[_{i+1}]{\Delta}{_{i,i+1}})^s}  \right) \right]^k F(G(x_0,x_1,x_1))\right] \\
                   \leq \ & K^{s}\left[ \left(\sum\limits_{k=n}^{n+p-1} \lambda^k \right) F(G(x_0,x_1,x_1))\right] \nonumber \\
                   \leq \ & K^{s}\frac{\lambda^n}{1-\lambda}F(G(x_0,x_1,x_1)) \nonumber .
\end{align}
As $n\to \infty$, since $F^{-1}(0)=\{0\}$ and $F$ is continuous, we deduce that $G(x_n,x_{n+p},x_{n+p}) \to 0.$ Thus $(x_n)$ is a $G$-Cauchy sequence.

Moreover, since $X$ is $G$-complete and $T$ is sequentially continuous, there exists $x^* \in X$ such that $(x_n)$ $G$-converges to $x^*$ and $(x_{n+1})$ $G$-converges to $Tx^*=x^*$ because $(X,G,K)$ is Hausdorff. The uniqueness of $x^*$ follows from the contractive condition  \ref{condition1}, since $\tensor*[_{k}]{\Gamma}{_{i,j}}<1.$ Indeed, for $x^*$ and $z^*$ fixed point of $T$, we have
\begin{align*}
F(G(x^*,z^*,z^*))&=F(G(T^ix^*,T^jz^*,T^kjz^*)) \\
& \leq  F(\tensor*[_{k}]{\Delta}{_{i,j}}[G(x^*,T^ix^*,T^ix^*)+ G(z^*,T^jz^*,T^jz^*) + G(z^*,T^kz^*,T^kz^*)])\\
                    &+ F(\tensor*[_{k}]{\Gamma}{_{i,j}}G(x^*,z^*,z^*) \\
            &< (\tensor*[_{k}]{\Gamma}{_{i,j}})^s G(x^*,z^*,z^*).       
\end{align*}

\end{proof}

As a particular case of Theorem \ref{lambda1}, we state the following corollary.

\begin{corollary}\label{lambda1cor}
Let $(X,G,K)$ be a $G$-complete $G$-metric type space, and let $T$ be a sequentially continuous self mapping on $X$. Assume that there exista a sequences $(a_n)$ of elements of $X$ such 

\begin{align}\label{condition1cor}
F(G(T^i(x),T^j(y),T^kj(z))) \leq \  
& F(\tensor*[_{k}]{\Delta}{_{i,j}}[G(x,T^ix,T^ix)+ G(y,T^jy,T^jy) + \nonumber \\
 &+ G(z,T^kz,T^kz)+G(x,y,z) ])
\end{align}

for $x,y,z\in X$ with $x\neq y,$ $0\leq \tensor*[_{k}]{\Delta}{_{i,j}}<\frac{1}{2} , \ i,j = 1,2,\cdots ,$ and some $F \in \Phi,$ homogeneous with degree $s$, where $\tensor*[_{k}]{\Delta}{_{i,j}}=G(a_i,a_j,a_k)$.
If the sequence $(r_i)$ where  $$r_i=\frac{(2\tensor*[_{i+1}]{\Delta}{_{i,i+1}})^s}{1-(2\tensor*[_{i+1}]{\Delta}{_{i,i+1}})^s}$$ is a non-increasing $\lambda$-sequence of $\mathbb{R}^+$ endowed with the $\max\footnote{The max metric $m$ refers to $m(x,y)=\max\{x,y\}$ }$ metric, then $T$ has a unique fixed point in $X$.
\end{corollary}

\begin{example}
Let $X=[0,1]$ and $G(x,y,z)= \max\{x,y,z \}$ whenever $x,y,z \in [0,1]$. Clearly, $(X,G,1)$ is a $G$-complete $G$-metric space. 

 Following the notation in Theorem \ref{lambda1cor}, we set $a_n = \left(\frac{1}{\sqrt{2}(1+2^n)}\right)^2$. 
 
 We also define $T(x)= \frac{x}{16}$ for all $x\in X$ and and $F:[0,\infty) \to [0,\infty), \ x\mapsto \sqrt{x}$. Then $F$ is continuous, non-decreasing, sub-additive and homogeneous of degree $s=\frac{1}{2}$ and $F^{-1}(0)=\{0\}$. Assume without loss of generality that $x> y\geq z$ and $i>j\geq k$. Hence we have

\[ F(G(T^ix,T^jy,T^kz)) = \sqrt{\frac{x^i}{16^i}}\leq \sqrt{\frac{x}{16^i}}\leq \sqrt{\frac{x}{16^k}}, \]

and
 
\[ F( \tensor*[_{k}]{\Delta}{_{i,j}}[G(x,Tx,Tx)+ G(y,Ty,Ty) + G(z,Tz,Tz)]  ) =  \sqrt { \left(\frac{1}{\sqrt{2}(1+2^k)}\right)^2 (2x+y+z)} .\]

Moreover, since $F$ is homogeneous of degree $s=\frac{1}{2}$, the sequence

$$r_i =\frac{(2\tensor*[_{i+1}]{\Delta}{_{i,i+1}})^s}{1-(2\tensor*[_{i+1}]{\Delta}{_{i,i+1}})^s}  =\frac{1}{2^i} $$

is a $\lambda$-sequence with $\lambda = 1/2.$

The conditions of Corollary \ref{lambda1cor} are satisfied, so $T$ has a unique fixed point, which in this case is $x^*=0.$

\end{example}

All the results of this section remain true if $X$ is $T$-orbitally complete and $T$ orbitally continuous.

\subsection{Variants of the $\lambda$-sequences hypothesis}
 \hspace{1cm}

\vspace*{0.1cm}

In this subsection, we give some fixed results that we shall not prove as their proofs follow immediately from the ones in the previous subsection.

\begin{theorem}\label{lambda1var}
Let $(X,G,K)$ be a $G$-complete $G$-metric type space, and let $T$ be a sequentially continuous self mapping on $X$. Assume that there exist two sequences $(a_n)$ and $(b_n)$ of elements of $X$ such 

\begin{align}\label{conditionvar1}
F(G(T^i(x),T^j(y),T^kj(z))) \leq \  
& F(\tensor*[_{k}]{\Delta}{_{i,j}}[G(x,T^ix,T^ix)+ G(y,T^jy,T^jy) + G(z,T^kz,T^kz)])\\
                    &+ F(\tensor*[_{k}]{\Gamma}{_{i,j}}G(x,y,z)) \nonumber
\end{align}

for $x,y,z\in X$ with $x\neq y,$ $0\leq \tensor*[_{k}]{\Delta}{_{i,j}},\tensor*[_{k}]{\Gamma}{_{i,j}}<1/2 , \ i,j = 1,2,\cdots ,$ and some $F \in \Phi,$ homogeneous with degree $s$, where $\tensor*[_{k}]{\Delta}{_{i,j}}=G(a_i,a_j,a_k)$ and $\tensor*[_{k}]{\Gamma}{_{i,j}}=G(b_i,b_j,b_k)$.
If the sequence $(r_i)$ where  $$r_i=\frac{(\tensor*[_{i+1}]{\Delta}{_{i,i+1}})^s+(\tensor*[_{i+1}]{\Gamma}{_{i,i+1}})^s}{1-2(\tensor*[_{i+1}]{\Delta}{_{i,i+1}})^s}$$ is 
such that :

\begin{itemize}
\item[i)] for each $j,k$, $\limsup_{i\to \infty} (\tensor*[_{k}]{\Delta}{_{i,j}})^s <1$ and $\limsup_{i\to \infty} (\tensor*[_{k}]{\Gamma}{_{i,j}})^s <1$,

\item[ii)] 
$$
\sum_{n=1}^{\infty} C_n <\infty \text{ where } C_n = \prod_{i=1}^{n}r_i = \prod_{i=1}^{n}\frac{(\tensor*[_{i+1}]{\Delta}{_{i,i+1}})^s+(\tensor*[_{i+1}]{\Gamma}{_{i,i+1}})^s}{1-(2\tensor*[_{i+1}]{\Delta}{_{i,i+1}})^s},
$$

\end{itemize}
 then $T$ has a unique fixed point in $X$.

\end{theorem}

\begin{theorem}\label{lambda2var}
Let $(X,G,K)$ be a $G$-complete $G$-metric type space, and let $T$ be a self mapping on $X$. Assume that there exist two sequences $(a_n)$ and $(b_n)$ of elements of $X$ such 

\begin{align}\label{conditionvar2}
F(G(T^i(x),T^j(y),T^kj(z))) \leq \  
& F(\tensor*[_{k}]{\Delta}{_{i,j}}[G(x,T^ix,T^ix)+ G(y,T^jy,T^jy) + G(z,T^kz,T^kz)])
\end{align}

for $x,y,z\in X$ with $x\neq y,$ $0\leq \tensor*[_{k}]{\Delta}{_{i,j}}<1/2 , \ i,j,k = 1,2,\cdots ,$ and some $F \in \Phi,$ homogeneous with degree $s$, where $\tensor*[_{k}]{\Delta}{_{i,j}}=G(a_i,a_j,a_k)$ and $\tensor*[_{k}]{\Gamma}{_{i,j}}=G(b_i,b_j,b_k)$.
If the sequence $(r_i)$ where  $$r_i=\frac{(2\tensor*[_{i+1}]{\Delta}{_{i,i+1}})^s}{1-2(\tensor*[_{i+1}]{\Delta}{_{i,i+1}})^s}$$ is 
such that :

\begin{itemize}
\item[i)] for each $j,k$, $\limsup_{i\to \infty} (\tensor*[_{k}]{\Delta}{_{i,j}})^s <1$,
\item[ii)] 
\[
\sum_{n=1}^{\infty} C_n <\infty \text{ where } C_n = \prod_{i=1}^{n}r_i = \prod_{i=1}^{n}\frac{(2\tensor*[_{i+1}]{\Delta}{_{i,i+1}})^s}{1-(2\tensor*[_{i+1}]{\Delta}{_{i,i+1}})^s},
\]

\end{itemize}

then $T$ has a unique fixed point in $X$.
\end{theorem}

\subsection{Common fixed points via $\lambda$-sequences}\hspace{1.2cm}

\vspace*{0.3cm}

The author plans to study more thoroughly and with examples common fixed point results in another paper \cite{Gaba6} but here is a first result of the kind.

\begin{theorem}\label{common1}

Let$X$ be a $G$-complete $G$-metric space $(X,G,K)$ and $\{T_n\}$ be a sequence of self mappings on $X$.  Assume that there exist three sequences $(a_n)$, $(b_n)$ and $(c_n)$ of elements of $X$ such that

\begin{align}\label{condcomon1}
G(T_ix,T_jy,T_kz) \ & \leq (\tensor*[_{k}]{\Delta}{_{i,j}})G(x,y,z) + (\tensor*[_{k}]{\Theta}{_{i,j}})[G(T_ix,x,x)+G(y,T_jy,y)+G(z,z,T_kz)] + \nonumber \\
&+  (\tensor*[_{k}]{\Lambda}{_{i,j}})[G(T_ix,y,z)+G(x,T_jy,z)+G(x,y,T_kz)],
\end{align}

for all $x,y,z\in X$ with $0\leq \tensor*[_{k}]{\Delta}{_{i,j}}+3 (\tensor*[_{k}]{\Theta}{_{i,j}})+4(\tensor*[_{k}]{\Lambda}{_{i,j}}) <1/2 , \ i,j,k = 1,2,\cdots ,$ and some $F \in \Phi,$ homogeneous with degree $s$, where $\tensor*[_{k}]{\Delta}{_{i,j}}=G(a_i,a_j,a_k)$, $\tensor*[_{k}]{\Theta}{_{i,j}}=G(b_i,b_j,b_k)$ and $\tensor*[_{k}]{\Lambda}{_{i,j}}=G(c_i,c_j,c_k)$. If the sequence $(r_i)$ where

$$r_i = \left[\frac{[(\tensor*[_{i+2}]{\Delta}{_{i,i+1}})+2(\tensor*[_{i+2}]{\Theta}{_{i,i+1}})+3(\tensor*[_{i+2}]{\Lambda}{_{i,i+1}})]}{1-(\tensor*[_{i+2}]{\Theta}{_{i,i+1}})- (\tensor*[_{i+2}]{\Lambda}{_{i,i+1}})}\right]$$

is a non-increasing $\lambda$-sequence of $\mathbb{R}^+$ endowed with the $\max\footnote{The max metric $m$ refers to $m(x,y)=\max\{x,y\}$ }$ metric, then $\{T_n\}$ have a unique common fixed point in $X$.

\end{theorem}

\begin{proof}
We will proceed in two main steps.

\vspace{0.3cm}

\underline{Claim 1:}  Any fixed point of $T_i$ is also a fixed point of $T_j$ and $T_k$ for $i\neq j\neq k\neq i$.

Assmue that $x^*$ is a fixed point of $T_i$ and suppose that $T_jx^*\neq x^*$ and $T_kx^*\neq x^*$. Then 

\begin{align*}
G(x^*,T_jx^*,T_kx^*) &= G(T_ix^*,T_jx^*,T_kx^*) \\
                 & \leq  (\tensor*[_{k}]{\Delta}{_{i,j}})G(x^*,x^*,x^*) 
                  + (\tensor*[_{k}]{\Theta}{_{i,j}})[G(T_ix^*,x^*,x^*) + G(x^*,T_jx^*,x^*)+ G(x^*,x^*,T_kx^*) ]\\
                 & + (\tensor*[_{k}]{\Lambda}{_{i,j}})[G(T_ix^*,x^*,x^*)+G(x^*,T_jx^*,x^*)+ G(x^*,x^*,T_kx^*) ] \\
                & \leq [ (\tensor*[_{k}]{\Theta}{_{i,j}}) + (\tensor*[_{k}]{\Lambda}{_{i,j}})][G(x^*,T_jx^*,T_kx^*)+G(x^*,T_jx^*,T_kx^*)]\\
                 & \leq [ (2\tensor*[_{k}]{\Theta}{_{i,j}}) + (2\tensor*[_{k}]{\Lambda}{_{i,j}})][G(x^*,T_jx^*,T_kx^*) ],
\end{align*}

which is a contradiction unless $T_ix^* = x^* =T_jx^*=T_kx^*.$

\vspace{0.3cm}

\underline{Claim 2:}

For any $x_0\in X$, we construct the sequence $(x_n)$ by setting $x_n= T_n(x_{n-1}),\ n=1,2,\cdots .$ We assume without loss of generality that $x_n\neq x_{n+1}$ for all $n\in \mathbb{N}$.
Using \eqref{condcomon1}, we obtain

\begin{align*}
G(x_1,x_2,x_3) & = G(T_1x_0,T_2x_1,T_3x_2) \\
               & \leq (\tensor*[_{3}]{\Delta}{_{1,2}})G(x_0,x_1,x_2) + 
               (\tensor*[_{3}]{\Theta}{_{1,2}})[G(x_1,x_0,x_0)+G(x_1,x_2,x_1)+G(x_2,x_2,x_3)]\\
               &+ (\tensor*[_{3}]{\Lambda}{_{1,2}})[ G(x_1,x_1,x_2)+ G(x_0,x_2,x_2)+G(x_0,x_1,x_3)].
\end{align*}


By property (G3), one can write

\begin{align*}
G(x_1,x_2,x_3) & = G(T_1x_0,T_2x_1,T_3x_2) \\
               & \leq (\tensor*[_{3}]{\Delta}{_{1,2}})G(x_0,x_1,x_2) + 
               (\tensor*[_{3}]{\Theta}{_{1,2}})[G(x_1,x_0,x_2)+G(x_1,x_2,x_0)+G(x_1,x_2,x_3)]\\
               &+ (\tensor*[_{3}]{\Lambda}{_{1,2}})[ G(x_1,x_0,x_2)+ G(x_0,x_1,x_2)+G(x_0,x_1,x_3)]
\end{align*}

Again since

\[
G(x_0,x_1,x_3) \leq G(x_0,x_2,x_2)+ G(x_2,x_1,x_3),
\]

we obtain, 

\begin{align*}
G(x_1,x_2,x_3) & = G(T_1x_0,T_2x_1,T_3x_2) \\
               & \leq (\tensor*[_{3}]{\Delta}{_{1,2}})G(x_0,x_1,x_2) + 
               (\tensor*[_{3}]{\Theta}{_{1,2}})[G(x_1,x_0,x_2)+G(x_1,x_2,x_0)+G(x_1,x_2,x_3)]\\
               &+ (\tensor*[_{3}]{\Lambda}{_{1,2}})[ G(x_1,x_0,x_2)+ G(x_0,x_1,x_2)+G(x_0,x_1,x_2)+ G(x_2,x_1,x_3)],
\end{align*}

that is 
\[
[1-(\tensor*[_{3}]{\Theta}{_{1,2}})- (\tensor*[_{3}]{\Lambda}{_{1,2}})]G(x_1,x_2,x_3) \leq [(\tensor*[_{3}]{\Delta}{_{1,2}})+2(\tensor*[_{3}]{\Theta}{_{1,2}})+3(\tensor*[_{3}]{\Lambda}{_{1,2}})]G(x_0,x_1,x_2).
\]

Hence

\[
G(x_1,x_2,x_3) \leq \frac{[(\tensor*[_{3}]{\Delta}{_{1,2}})+2(\tensor*[_{3}]{\Theta}{_{1,2}})+3(\tensor*[_{3}]{\Lambda}{_{1,2}})]}{1-(\tensor*[_{3}]{\Theta}{_{1,2}})- (\tensor*[_{3}]{\Lambda}{_{1,2}})}G(x_0,x_1,x_2).
\]

Also we get

\begin{align*}
G(x_2,x_3,x_4) &\leq \frac{[(\tensor*[_{4}]{\Delta}{_{2,3}})+2(\tensor*[_{4}]{\Theta}{_{2,3}})+3(\tensor*[_{4}]{\Lambda}{_{2,3}})]}{1-(\tensor*[_{4}]{\Theta}{_{2,3}})- (\tensor*[_{4}]{\Lambda}{_{2,3}})}G(x_1,x_2,x_3)\\
            &\leq  \left[ \frac{[(\tensor*[_{4}]{\Delta}{_{2,3}})+2(\tensor*[_{4}]{\Theta}{_{2,3}})+3(\tensor*[_{4}]{\Lambda}{_{2,3}})]}{1-(\tensor*[_{4}]{\Theta}{_{2,3}})- (\tensor*[_{4}]{\Lambda}{_{2,3}})}\right] \left[ \frac{[(\tensor*[_{3}]{\Delta}{_{1,2}})+2(\tensor*[_{3}]{\Theta}{_{1,2}})+3(\tensor*[_{3}]{\Lambda}{_{1,2}})]}{1-(\tensor*[_{3}]{\Theta}{_{1,2}})- (\tensor*[_{3}]{\Lambda}{_{1,2}})}\right]G(x_0,x_1,x_2).
\end{align*}

Repeating the above reasoning, we obtain
\[
G(x_{n},x_{n+1},x_{n+2}) \leq \prod\limits_{i=1}^n \left[\frac{[(\tensor*[_{i+2}]{\Delta}{_{i,i+1}})+2(\tensor*[_{i+2}]{\Theta}{_{i,i+1}})+3(\tensor*[_{i+2}]{\Lambda}{_{i,i+1}})]}{1-(\tensor*[_{i+2}]{\Theta}{_{i,i+1}})- (\tensor*[_{i+2}]{\Lambda}{_{i,i+1}})}\right]G(x_0,x_1,x_2)
\]

If we set $$r_i  = \left[\frac{[(\tensor*[_{i+2}]{\Delta}{_{i,i+1}})+2(\tensor*[_{i+2}]{\Theta}{_{i,i+1}})+3(\tensor*[_{i+2}]{\Lambda}{_{i,i+1}})]}{1-(\tensor*[_{i+2}]{\Theta}{_{i,i+1}})- (\tensor*[_{i+2}]{\Lambda}{_{i,i+1}})}\right], $$

we have that 

$$G(x_{n},x_{n+1},x_{n+2}) \leq \left[ \prod\limits_{i=1}^n r_i \right] G(x_0,x_1,x_2). $$


Therefore, for all $l>m>n$

\begin{align*}
G(x_n,x_{m},x_{l}) & \leq G(x_{n},x_{n+1},x_{n+1})+  G(x_{n+1},x_{n+2},x_{n+2}) + \\
 & + \cdots + G(x_{l-1},x_{l-1},x_{l}) \\ 
 & \leq G(x_{n},x_{n+1},x_{n+2})+  G(x_{n+1},x_{n+2},x_{n+3}) + \\
 & + \cdots + G(x_{l-2},x_{l-1},x_{l}),
\end{align*}


and

\begin{align*}
G(x_n,x_{m},x_{l})& \leq \left(\left[ \prod\limits_{i=1}^n r_i \right] + 
\left[ \prod\limits_{i=1}^{n+1} r_i \right]+ \cdots + \left[ \prod\limits_{i=1}^{n+l-1} r_i \right]\right) G(x_0,x_1,x_2) \\
& = \sum_{k=0}^{l-1}\left[\prod\limits_{i=1}^{n+k}r_i\right]G(x_0,x_1,x_2)\\
& =\sum_{k=n}^{n+l-1}\left[\prod\limits_{i=1}^{k}r_i\right]G(x_0,x_1,x_2).
\end{align*}

Now, let $\lambda$ and $n(\lambda)$ as in Definition \ref{def1}, then for $n\geq n(\lambda)$ and using the fact that the geometric mean of non-negative real numbers is at most their arithmetic mean, it follows that

\begin{align*}
G(x_n,x_{m},x_{l})& \leq \sum_{k=n}^{n+l-1}\left[\frac{1}{k}\left(\sum_{i=1}^{k}r_i\right)\right]^kG(x_0,x_1,x_2)\\
& =\left(\sum_{k=n}^{n+l-1} \alpha^k\right) G(x_0,x_1,x_2)\\
& \leq \frac{\alpha^n}{1-\alpha}G(x_0,x_1,x_2).
\end{align*}

As $n\to \infty$, we deduce that $G(x_n,x_{m},x_{l}) \to 0.$ Thus $(x_n)$ is a $G$-Cauchy sequence.
Moreover, since $X$ is $G$-complete there exists $u \in X$ such that $(x_n)$ $G$-converges to $u$.

If there exists $n_0$ such that $T_{n_0}u=u$, then by the claim 1, the proof of existence is complete.

Otherwise for any positive integers $k, l$, we have
\begin{align*}
G(x_n, T_ku,T_lu)  & = G(T_nx_{n-1},T_ku,T_lu) \\
                   & \leq (\tensor*[_{l}]{\Delta}{_{n,k}})G(x_n,u,u) + (\tensor*[_{l}]{\Theta}{_{n,k}})[G(T_nx_{n-1},x_{n-1},x_{n-1})+G(u,T_ku,u)+G(u,u,T_lu)]\\
                   &+ (\tensor*[_{l}]{\Lambda}{_{n,k}})[G(T_nx_{n-1},u,u)+G(x_{n-1},T_ku,u)+G(x_{n-1},u,T_lu)]
\end{align*}

Letting $n\to \infty$, and using property (G3) we obtain

\begin{align*}
G(u, T_ku,T_lu)  & \leq (\tensor*[_{l}]{\Theta}{_{n,k}})[G(u,T_ku,u)+G(u,u,T_lu)] \\
& + (\tensor*[_{l}]{\Lambda}{_{n,k}})[G(u,T_ku,u)+G(u,u,T_lu)]\\
& \leq  [ (2\tensor*[_{k}]{\Theta}{_{i,j}}) + (2\tensor*[_{k}]{\Lambda}{_{i,j}})][G(u,T_ku,T_lu)+G(u,T_ku,T_lu)]
\end{align*}
and this is a contradiction, unless $ u=T_ku=T_lu$.

Finally, we prove the uniqueness of the
common fixed point $u$. To this aim, let us suppose that $v$ is another common fixed
point of ${T_m}$, that is, $T_m(v) = v, \ \forall m\geq 1$. Then, using \ref{condcomon1}, we have

\[
G(u, v,v)= G(T_nu, T_kv,T_lv)   \leq (\tensor*[_{l}]{\Delta}{_{n,k}})G(u,v,v)+3 (\tensor*[_{l}]{\Lambda}{_{n,k}}) G(u,v,v)<  G(u,v,v),
\]
which yields $u=v$. So, $u$ is the unique common fixed point of $\{T_m\}$.
\end{proof}

We conclude this manuscript with the following result, whose proof is straightforward, following the steps of the proofs of Theorem \ref{lambda1} and Theorem \ref{common1} .

\begin{theorem}\label{lambdafin}
Let $X$ be a $G$-complete $G$-metric space $(X,G,K)$ and $\{T_n\}$ be a sequence of self mappings on $X$.  Assume that there exist three sequences $(a_n)$, $(b_n)$ and $(c_n)$ of elements of $X$ such that

\begin{align}\label{conditionfin}
F[G(T_ix,T_jy,T_kz)] \ & \leq F[(\tensor*[_{k}]{\Delta}{_{i,j}})G(x,y,z) + (\tensor*[_{k}]{\Theta}{_{i,j}})[G(T_ix,x,x)+G(y,T_jy,y)+G(z,z,T_kz)] + \nonumber \\
&+  (\tensor*[_{k}]{\Lambda}{_{i,j}})[G(T_ix,y,z)+G(x,T_jy,z)+G(x,y,T_kz)]],
\end{align}

for all $x,y,z\in X$ with $0\leq (\tensor*[_{k}]{\Delta}{_{i,j}})^s+3 (\tensor*[_{k}]{\Theta}{_{i,j}})^s+4(\tensor*[_{k}]{\Lambda}{_{i,j}})^s <1/2 , \ i,j,k = 1,2,\cdots ,$ and some $F \in \Phi,$ homogeneous with degree $s$, where $\tensor*[_{k}]{\Delta}{_{i,j}}=G(a_i,a_j,a_k)$, $\tensor*[_{k}]{\Theta}{_{i,j}}=G(b_i,b_j,b_k)$ and $\tensor*[_{k}]{\Lambda}{_{i,j}}=G(c_i,c_j,c_k)$. If the sequence $(r_i)$ where

$$r_i = \left[\frac{[(\tensor*[_{i+2}]{\Delta}{_{i,i+1}})^s+2(\tensor*[_{i+2}]{\Theta}{_{i,i+1}})^s+3(\tensor*[_{i+2}]{\Lambda}{_{i,i+1}})^s]}{1-(\tensor*[_{i+2}]{\Theta}{_{i,i+1}})^s- (\tensor*[_{i+2}]{\Lambda}{_{i,i+1}})^s}\right]$$

is a non-increasing $\lambda$-sequence of $\mathbb{R}^+$ endowed with the $\max\footnote{The max metric $m$ refers to $m(x,y)=\max\{x,y\}$ }$ metric, then $\{T_n\}$ have a unique common fixed point in $X$.

\end{theorem}

\bibliographystyle{amsplain}

\end{document}